\documentclass[12pt]{article}
\usepackage{srcltx}
\usepackage{amsfonts,amssymb,mathrsfs,amsmath,cmtiup}
\usepackage{amsthm}

\newtheorem{theorem}{Theorem}[section]
\newtheorem{lemma}[theorem]{Lemma}

\theoremstyle{definition}

\newtheorem*{Index Convention}{Index Convention}

\def\keywords#1{\par\medskip
\noindent\textbf{Keywords.} #1}

\def\subjclass#1{{\renewcommand{\thefootnote}{}
\footnote{\emph{Mathematics Subject Classification (2010):} #1}}}

\begin{document}
\let\le=\leqslant
\let\ge=\geqslant
\let\leq=\leqslant
\let\geq=\geqslant
\newcommand{\e}{\varepsilon }
\newcommand{\f}{\varphi }
\newcommand{ \g}{\gamma}
\newcommand{\F}{{\Bbb F}}
\newcommand{\N}{{\Bbb N}}
\newcommand{\Z}{{\Bbb Z}}
\newcommand{\Q}{{\Bbb Q}}
\newcommand{\C}{{\Bbb C}}
\newcommand{\R}{\Rightarrow }
\newcommand{\W}{\Omega }
\newcommand{\w}{\omega }
\newcommand{\s}{\sigma }
\newcommand{\hs}{\hskip0.2ex }
\newcommand{\ep}{\makebox[1em]{}\nobreak\hfill $\square$\vskip2ex }
\newcommand{\Lr}{\Leftrightarrow }
\sloppy

\title{Almost nilpotency of an  associative algebra with \\  an almost nilpotent fixed-point subalgebra}


\author{{\sc N.\,Yu.~Makarenko \\ \small Sobolev Institute of Mathematics, Novosibirsk,
630\,090,
Russia, \\[-1ex] \small natalia\_makarenko@yahoo.fr
}}

\author{
{N.\,Yu.~Makarenko\footnote{The research is supported by RSF (project N 14-21-00065)}}\\
\small  Sobolev Institute of Mathematics, Novosibirsk, 630\,090,
Russia
\\[-1ex] \small  natalia\_makarenko@yahoo.fr}

\date{}
\maketitle




\begin{abstract}

 Let $A$ be an associative algebra of arbitrary dimension over a
field $F$ and   $G$ a finite group of automorphisms of $A$ of
order $n$, prime to the characteristic of $F$.  Denote by
 $A^G=\{ a\in A\,\mid \, a^g=a\,\, \mathrm{for\,\, all}\,\, g \in G\}$
 the fixed-point
 subalgebra.  By the classical Isaacs--Bergman theorem,
 if~$A^G$ is nilpotent of index $d$, i.e.
$(A^G)^d=0$, then $A$ is also nilpotent and its nilpotency index
is bounded by a function depending only on $n$ and $d$.  We prove,
under additional assumption of solubility of $G$, that if $A^G$
contains a two-sided nilpotent ideal $I\lhd A^G$ of nilpotency
index $d$ and of finite codimension $m$, then $A$ contains a
nilpotent two-sided ideal $H\lhd A$ of nilpotency index bounded by
a function of $n$ and $d$ and of finite codimension bounded by a
function of $m$, $n$ and~$d$. An even stronger result is provided
for graded associative algebras: if $G$ is a finite (not
necessarily soluble) group of order $n$ and
  $A= \bigoplus_{g\in G}A_g$ is a $G$-graded
 associative algebra over a field $F$, i.e. $A_gA_h\subset
A_{gh}$, such that  the identity component~$A_e$ has a two-sided
nilpotent ideal $I_e\lhd A^G$  of index $d$ and of finite
codimension $m$ in $A_e$, then $A$ has a homogeneous nilpotent
two-sided ideal $H\lhd A$ of index bounded by a function on $n$
and $d$ and of finite codimension bounded by a function on $n$,
$d$ and $m$.
\end{abstract}



 \section{Introduction}\label{intro}

By the classical Isaacs--Bergman theorem~\cite{be-is}, if an
associative algebra $A$  over a field $F$ admits a finite group of
automorphisms $G$ of order $n$, prime to the characteristic of
$F$, and the fixed-point subalgebra  $A^G=\{ a\in A\,\mid \,
a^g=a\,\, \mathrm{for\,\, all}\,\, g \in G\}$ is nilpotent of
index~$d$, i.e. $(A^G)^d=0$, then $A$ is nilpotent of index
bounded by a function of $n$ and $d$. Starting from this work, a
great number of paper was devoted to study of properties of
algebras (or rings) subject to corresponding properties of
fixed-point algebra under finite group actions. In this paper we
prove, under the additional assumption of the solubility of the
automorphism group, that the ``almost nilpotency'' of the
fixed-point subalgebra implies the ``almost nilpotency'' of the
algebra itself. Namely, the following theorem holds.

\begin{theorem} \label{th1} Let $A$ be an associative algebra of arbitrary
(possibly infinite)  dimension over a field $F$ acted on by a
finite soluble  group $G$ of order $n$. Suppose that the
characteristic of $F$ does not divide $n$.  If the fixed-point
subalgebra $A^G$ has a nilpotent two-sided ideal $I\lhd A^G$ of
nilpotency index $d$ and of finite codimension $m$ in $A^G$, then
$A$ has a nilpotent two-sided ideal $H\lhd A$ of nilpotency index
bounded by a function of $n$ and $d$ and of finite codimension
bounded by a function of $m$, $n$  and~$d$.
\end{theorem}


The restrictions on  the order of the automorphism group are
unavoidable. There are examples showing that no results of this
kind are possible either for infinite automorphism groups  or for
algebras with $n$-torsion.

\vskip2ex

Theorem~\ref{th1} follows by induction on the order of $G$ from
the Bergman-Isaacs theorem and the following statement on graded
associative algebras, in which we do not suppose either $G$ to be
soluble or the order of $G$ to be prime to the characteristic of
the field.

\begin{theorem} \label{th2} Let $G$ be a finite group of order $n$ and
 let  $A= \bigoplus_{g\in G}A_g$ be a $G$-graded
 associative algebra over a field $F$, i.e. $A_gA_h\subset
A_{gh}$.  If the identity component $A_e$ has a nilpotent
two-sided ideal $I_e\lhd A_e$ of nilpotency index $d$ and  of
finite codimension $m$ in $A_e$, then $A$ has a homogeneous
nilpotent two-sided ideal $H\rhd A$ of nilpotency index bounded by
a function on $n$ and $d$ and of finite codimension bounded by a
function on $n$, $d$ and $m$.
\end{theorem}

The proof of Theorem~\ref{th2} is based on  the method of
generalized centralizers, originally  created by Khukhro in
\cite{kh1} for nilpotent groups and Lie algebras with an almost
regular automorphism of prime order. In \cite{ma-kh3, ma-kh4} the
approach was significantly revised and new techniques were
introduced to study a more complicated case of an almost regular
automorphism of arbitrary (not necessarily prime) finite order. In
particular, it was proved that if a Lie algebra $L$ admits an
automorphism $\varphi$ of finite order $n$ with finite-dimensional
fixed-point subalgebra of dimension $\rm {dim}\,C_L(\varphi)=m$,
then $L$ has a soluble ideal of derived length bounded by a
function of $n$ whose codimension is bounded by a function of $ m$
and~$n$.
  The fundamental combinatorial
nature of the construction  in~\cite{ma-kh4} makes possible to
apply it to a wide range of situations. For example, the approach
was used to study  Lie type algebras (a large class of algebras
which includes associative, Lie algebras, color Lie superalgebras)
with an almost regular automorphism of finite order in~\cite{ma3}.

In the proof of Theorem \ref{th2} we  use virtually the same
construction as in~\cite{ma-kh4}. But the strong condition of
associativity simplifies the reasoning and allows to provide much
stronger results than in the case of Lie algebras. In particular,
we do not need to suppose that the automorphism group is cyclic.

We give some definitions and auxiliary  lemmas  in~\S\,2.
 In~\S\,3 we prove Theorem~\ref{th2}. For this, we set $N=d^2+1$ and for each $g\in G \setminus {e}$ we construct by induction
generalized centralizers $A_g(i)$ of levels $i=1, 2, \ldots N$,
which  are some subspaces of the homogeneous components $A_g$.
Then we demonstrate that the ideal generated by all the $A_g(N)$,
$g\in G \setminus {e}$ is the required one. In~\S\,4 by induction
on the order of $G$ we derive the Theorem~\ref{th1} from
Theorem~\ref{th2} and the Bergman-Isaacs Theorem.

Throughout the paper we will say that a number is   ``$(a, b,
\ldots )$- bounded'' if it is  ``bounded above by some function
depending only on $a, b, \ldots$''.

\section {Preliminaries}

If $G$ is a group of automorphisms of $A$, then $A^G=\{ a\in
A\,\mid \, a^g=a\,\, \mathrm{for\,\, all}\,\, g \in G\}$ will
denote the fixed-point subalgebra. The two-sided ideal $H$ of $A$
is denoted by $H\lhd A$. If $I$ and $J$ are subspaces of $A$, $IJ$
will denote the subspace spanned by all products $ab$ with $a \in
I$ and $b \in J$, and $I^d$ will denote the $d$-fold product
$\underbrace{I\ldots I}_{d}$.
 We say
that an algebra is nilpotent of (nilpotency) index $d$ if the
product of any $d$ elements of the ring $A$ equals zero, i.e. $A^d
= 0$. The  subalgebra  generated by subspaces~$B_1,B_2,\ldots,
B_s$ is denoted by $\left<B_1,B_2,\ldots, B_s\right>$, and the
two-sided ideal generated by~$B_1,B_2,\ldots, B_s$ is denoted by
${}_{\rm id}\!\left<B_1,B_2,\ldots, B_s\right>$.
 If $H$ is an algebra, then $H^{\#}$ will denote the algebra obtained by adjoining
 $1$ to $H$. The (two-sided) ideal of $H$ generated by a subspace $I$ is sometimes written as $H^{\#}IH^{\#}$.

 \vskip2ex
We now state some facts needed in the proof.

\begin{lemma}\label{l-b-i} (Bergman--Isaacs Theorem \cite{be-is}). Let $G$ be a
finite group of automorphisms of an associative ring (algebra) $R$
of order $n$. If $R$ has no $n$-torsion and $R^G$ is nilpotent  of
index $d$ then $R$ is nilpotent of index at most $h^d,$ where $h=
1+\prod_{i=0}^n(C_n^i+1)$.
\end{lemma}

The following two lemmas are known. We give their proofs for the
convenience of readers.

\begin{lemma} \label{l-b-i-2} (Bergman--Isaacs \cite[lemma 1.1]{be-is}). Let $G$ be a finite group of order $n$ and
 let  $A= \bigoplus_{g\in G}A_g$ be a $G$-graded
 associative algebra over a field $F$, i.e. $A_gA_h\subset
A_{gh}$.  If the identity component $A_e$ is nilpotent of index
$d$, then $A$ is nilpotent of  index at most $nd$.
\end{lemma}
\begin{proof} It suffices to prove that a product $a_{g_1}a_{g_2}\ldots a_{g_{nd}}$ in homogeneous
elements $a_{g_i}\in A_{g_i}$, $i=1,\ldots, nd$ of length $nd$ is
trivial. We consider  $nd+1$ products $h_0=e$, $h_1=g_1$, $h_i =
g_1\ldots g_i$, $i=1,\ldots, nd$. Since the order of $G$ is $n$,
some $d+1$ elements  must be equal. If $h_i=h_j$ with $i <j$, then
$g_{i+1}\ldots g_j=e$. We obtain that $a_{g_1}a_{g_2}\ldots
a_{g_{nd}}$ can be represented as $P_1Q_1Q_2\ldots Q_d P_2$, where
$P_1$ and $P_2$ are (possibly empty) products in homogeneous
$a_{g_i}$, and each $Q_i$ is a non-empty product of the form $Q_i=
 a_{g_{i+1}}\ldots a_{g_{j}}$ with $g_{i+1}\ldots g_j=e$. It follows that $Q_i\in A_e$ for all $i=1,\ldots,d$. Since $(A_e)^d=0$,
 we have that $Q_1Q_2\ldots Q_d=0$, and therefore $a_{g_1}a_{g_2}\ldots
a_{g_{nd}}=0$.
\end{proof}

\begin{lemma} [{\cite[Lemma
1.3.7 ]{Kharchenko}}]\label{l4-2} Let $A$ be an associative
algebra over a field $F$ acted on by a finite group $G$ of order
$n$. Suppose that the characteristic of $F$ does not divide $n$.
If the subalgebra of invariants $A^G$ contains a nilpotent ideal
$I\lhd A^G$ of nilpotency index $d$, then $A$ contains a
$G$-invariant ideal $J\geq I$   of $(n,d)$-bounded nilpotency
index.
\end{lemma}
\begin{proof} Consider the right-sided ideal $B=IA^{\#}$
generated by $I$ (recall that  $A^{\#}$ is the algebra obtained
from $A$ by joining the unit). Let $b\in B^G$ be an element of $B$
fixed by~$G$, i. e. $b^g=b$ for all $g\in G$. There exist elements
$s\in B$, $i_m\in I$, $a_m\in A^{\#}$ such that $b=ns=n\sum_m i_m
a_m$. Since $\sum_{ g\in G} a_m^g\in A^G$ and $I$ is an ideal in
$A^G$, we have
$$b=ns=\sum_{ g\in G}s^g=\sum_m i_m \sum_{ g\in G} a_m^g\in I,$$
i. e. $B^G\leq I$ and $B^G$ is nilpotent of index $\leq d$.
Applying Bergman--Isaacs Theorem to the algebra $B$ we obtain that
$B$ is nilpotent of index at most $h^d,$ where $h=
1+\prod_{i=0}^n(C_n^i+1)$.  Finally, two-sided $G$-invariant ideal
$J\geq I$ generated  by $B$ is also nilpotent of index at most
$h^d$: $(A^{\#}IA^{\#})^{h^d}=A^{\#}(IA^{\#})^{h^d}=0.$
\end{proof}

\section {Proof of Theorem \ref{th2}}

Let $G$ be an arbitrary finite group of order $n$ and
  $A= \bigoplus_{g\in G}A_g$ be a $G$-graded
 associative algebra over a field $F$, i.e. $A_gA_h\subset
A_{gh}$.  Suppose that  the identity component $A_e$ has a
nilpotent ideal $I_e$ of nilpotency index $d$  and $\mathrm{dim}\,
A_e/I_e=m$.

\vskip2ex

{\bf Index Convention.}  In what follows, unless otherwise stated,
a small letter with an index $g$ will denote an element of the
homogeneous component $A_g$. The index  only indicates which
component  this element belongs to: $x_g\in A_g$. To lighten the
notation, we shall not be using numbering indices for elements of
the $A_g$, so that different elements can be denoted by the same
symbol. For example, $x_{g}$ and $x_{g}$ can be different elements
of $A_{g}$.

\vskip2ex

{\bf Construction of generalized centralizers and
representatives.}  We fix $N=d^2+3$. In each homogeneous component
$A_g$, $g\in G\setminus \{e\}$ we construct by induction a
descending chain of subspaces:
$$A_g=A_g(0)\geq A_g(1)\geq \cdots \geq A_g(N).$$
The subspaces $A_g(s)$ are called generalized centralizers of
level $s$. Simultaneously we fixe some homogeneous elements in
$A_g(s)$, $s=0,\ldots, N$ which are referred to as representatives
of level $s$. The total number of representatives will be $(n, d,
m)$-bounded.

\vskip2ex

{\bf Definition.}   For a monomial $a_{g_1}a_{g_2}\ldots a_{g_k},$
where $a_{g_i}\in A_{g_i}$, the record $(\ast_{g_1}
\ast_{g_2}\cdots \,\ast_{g_k})$ is called the {\it pattern\/} of
the monomial. The {\it length} of a pattern is the degree of the
monomial. The monomial is said to be the {\it value of its
pattern} on the given elements.

 For example, $a_{g}a_g a_v$ and
$b_{g}c_g b_v$ are values of the same  pattern $(\ast_{g}\,
\ast_{g}\,\ast_{v})$. (Under the Index Convention the elements
$a_{g}$ in the first product can be different.) \vskip2ex

 {\bf Definition.}
 Let $g\in G\setminus\{e\}$. For every ordered tuple
of elements $\vec x=(x_{g_1},\ldots ,x_{g_k}),$
 $x_{g_s}\in A_{g_s}$,  such that
$g_1g_2\ldots g_{l-1}\,g\, g_l\ldots g_k=e$ for some
$l\in\{1,\ldots k+1\}$ we define the mappings:
$$\vartheta _{\vec x,l}:\, \, A_g\rightarrow  A_e/I_e;$$
$$\vartheta _{\vec x,l}:\, \, y_g\rightarrow  x_{g_1} x_{g_2}\ldots x_{g_{l-1}} y_g x_{g_l} \ldots
x_{g_k}I_e,$$  where  $I_e$ is the nilpotent ideal of $A_e$ of
nilpotency index $d$ and of codimension $m$ in~$A_e$. We use index
$l$ to distinguish eventual  cases of $g_1g_2\ldots g_{k-1}\,g\,
g_k\ldots g_k=e$ and $g_1g_2\ldots g_{l-1}\,g\, g_l\ldots g_k=e$
with $k\neq l$ which lead to different mappings.

By linearity,  the mapping $\vartheta _{\vec x,l}$ is a
homomorphism of the subspace $A_g$ into factor-space $A_e/I_e$.
Since  ${\rm dim}\, A_e/I_e\leq m$, we have ${\rm dim}\,
A_g/\mbox{Ker}\, \vartheta _{\vec x,l}\leq m$. \vskip2ex

{\bf Definition of level  0.} We set $A_g(0)=A_g$ for all $g\in
G\setminus \{e\}$.  To construct representatives of level $0$ we
fix some elements $x_e\in A_e$ whose images form a basis of
$A_e/I_e$. These elements are  called {\it representatives of
level}
 $0$ and are denoted  by $x_{e}(0)$ (under the Index Convention).
 In addition we
consider  a pattern ${\bf P}= (\ast_g\, \ast_{g^{-1}})$ of length
$2$ with $g\in G\setminus \{e\}$. The dimension of the subspace of
the factor-space $A_e/I_e$ spanned by all images of values
 of ${\bf P}$ on  homogeneous elements of~$A_{g}, \, A_{g^{-1}}$ is at most $m$ by hypothesis.
 Hence we can
  choose at most $m$  products $c=x_gx_{g^{-1}}\in A_e$ whose
  images
 form a basis of this subspace.  The elements $x_g, x_{g^{-1}}$
  involved in these representations of the elements $c$ are also called  {\it representatives of level}
 $0$ and are denoted  by $x_{g}(0),\, x_{g^{-1}}(0)$ (under the Index Convention).
 The same is done for every pattern ${\bf
 P}$  of the form $(\ast_g\,
\ast_{g^{-1}})$, $g\in G\setminus \{e\}$.

Since ${\rm dim}\, A_e/I_e\leq m$  and the total number of
patterns ${\bf P}$ is $n-1$, the number of representatives of
level $0$ is at most
 $2(n-1)m+m$.

 \vskip2ex

{\bf Definition of level $\pmb{1}$.} Let $W_1=2d^3(n-1)+2$.  For
each  $g\in G \setminus \{e\}$ we set
$$A_g(1)=\bigcap_{\vec z} \bigcap_{l}\,\mbox{Ker}\, \vartheta_{\vec z,l},$$
where $\vec z=\left( z_{g_1}(0),\, \ldots ,\, z_{g_k}(0)\right) $
runs over all possible ordered tuples of all lengths
 $k\leq W_1$ consisting of representatives of  level $0$
  such that
 $ g_1\ldots g \ldots g_k=e$; if for a fixed tuple
 $\vec z=\left( z_{g_1}(0),\, \ldots ,\, z_{g_k}(0)\right) $ of length $k$ there are several different integers $l\leq k+1$  such that
 $ g_1\ldots g_{l-1}g g_{l}\ldots g_k=e$, we take the intersection over all such  integers $l$.
 The subspaces $A_g(1)$ are referred to as the  {\it generalized centralizers of level
 $1$}, elements of the  $A_g(1)$ are called {\it centralizers of
level $ 1$} and
 are denoted  by
 $y_g(1)$ (under the Index Convention).

The subspace $A_g(1)$  has $(n,d,m)$-bounded codimension
in~$A_{g}$ since the intersection here is taken over an $ (n,d,m)
$-bounded number of subspaces of $m$-bounded codimension in $A_g$.

The representatives of level~$1$ are constructed in two different
ways. First, for each
 $g\in G \setminus \{e\}$ we fix some elements of $A_g$ whose images form a
basis of of the factor-space
 $A_g/A_g(1)$. These elements  are called
{\it $b$-representatives of level $1$} are denoted by
 $ b_g(1)\in A_g$ (under the Index Convention). Since the dimensions  $A_g/A_g(1)$ are $(n,d,m)$-bounded for all
$g \in G\setminus \{e\}$,
  the total number of
 $b$-representatives of level $1$ is
 $(n,d,m)$-bounded.

Second, for each pattern ${\bf P}=(\ast_g \, \ast_{g^{-1}})$ of
length $2$ with indices  $g, g^{-1}\in G\setminus \{e\}$ we
consider the subspace of the factor-space $A_e/I_e$ spanned by all
images of  the values of ${\bf P}$ on homogeneous elements  of
$A_{g}(1),\, A_{g^{-1}}(1)$. Since ${\rm dim}\, A_e/I_e\leq m$, we
can choose at most $m$ products $c=y_g(1) y_{g^{-1}}(1)\in A_e$
whose images form a basis of this subspace in  $A_e/I_e$ and fix
the elements $y_g(1),\,y_{g^{-1}}(1)$
  involved in these representations. These elements
   are called {\it $x$-representatives of level $1$} and are
denoted by $ x_{g}(1)$ (under the Index Condition). Since the
number of patterns under consideration is equal to $n-1$, the
total number of $x$-representatives of level  $1$ is at most
$2(n-1)m$.

By construction, if $ g_1\ldots g_{t-1}\,g\, g_{t} \ldots g_k=e$,
for some $t\leq k+1$ and   $k\leq W_1$, a centralizer  $y_g(1)$
has the following property with respect to representatives
$x_{g_j}(0)$
 of level $0$:
\begin{equation}x_{g_1}(0)\, \ldots\,
x_{g_{t-1}}(0)\,y_g(1)\,x_{g_t}(0)\,\ldots\,
 x_{g_k}(0)\in I_e.
\end{equation}
 \vskip2ex

{\bf Definition of level $\pmb{s>0}$.}  Suppose that we have
already fixed representatives of level $<s$, which are
 either $x$-representatives  or  $b$-representatives and its
 number is $(m,n,d)$-bounded.
 We now define the  {\it generalized centralizers of level $s$}. Let $W_s=W_{s-1}+1=2d^3(n-1)+1+s$.  For each  $g\in G \setminus
 \{e\}$ we set
$$A_g(s)=\bigcap_{\vec z}\,\bigcap_{l}\,\mbox{Ker}\, \vartheta_{\vec z,l},$$
where $\vec z=\left( z_{g_1}(\varepsilon_1),\,\ldots, \,
z_{g_k}(\varepsilon_k)\right) $ runs over all possible ordered
tuples of all lengths
 $k\leq W_s$ consisting of representatives of (possibly different) levels $<s$
 (i.~e.,  $z_{g_u}(\varepsilon_u)$ denote elements of the form
 $x_{g_u}(\varepsilon_u) $ or $b_{g_u}(\varepsilon_u)$, $\varepsilon_u<s$,
 in
 any combination) such that
 $$ g_1\ldots g  \ldots g_k=e;$$  if for a fixed tuple
 $\vec z=\left( z_{g_1}(\varepsilon_1),\, \ldots ,\, z_{g_k}(\varepsilon_k)\right) $ of length $k$ there are several different integers $l\leq k+1$  such that
 $ g_1\ldots g_{l-1}g g_{l}\ldots g_k=e$, we take the intersection over all such  integers~$l$. Elements of the  $A_g(s)$ are also called  {\it
centralizers of level $ s$} and
 are denoted  by
 $y_g(s)$ (under the Index Convention).

The intersection here is taken over an  $ (n,d,m) $-bounded number
of subspaces of $m$-bounded codimension in  $A_g$, since the
number of representatives of all levels $<s$ is
 \ $(n,d,m)$-bounded and  ${\rm dim}\,A_g/\mbox{Ker}\, \vartheta _{\vec
z,l}\leq m$ for all $\vec z$. Hence
 $A_g(s)$ also has $(n,d,m)$-bounded codimension in the subspace~$A_{g}$.

We now fix representatives of level~$s$. First, for each
 $g\in G \setminus \{e\}$ we fix some elements of $A_g$ whose images form a
basis of of the factor-space
 $A_g/A_g(s)$.
These elements are denoted by
 $ b_g(s)\in A_g$ (under the Index Convention) and  are called
{\it $b$-representatives of level $s$}. The total number of
 $b$-representatives of level $s$ is
 $(n,d,m)$-bounded, since the dimensions  $A_g/A_g(s)$ are
$(n,d,m)$-bounded for all $g \in G\setminus \{e\}$.

Second, for each pattern ${\bf P}=(\ast_g \, \ast_{g^{-1}})$ of
length $2$ with indices  $g\in G\setminus \{e\}$ we consider the
subspace of the factor-space $A_e/I_e$ spanned by all images of
values of ${\bf P}$ on homogeneous elements  of $A_{g}(s),\,
A_{g^{-1}}(s)$. Since ${\rm dim}\, A_e/I_e\leq m$, we can choose
at most $m$ products $c=y_g(s) y_{g^{-1}}(s)\in A_e$ whose images
form a basis of this subspace in  $A_e/I_e$ and fix the elements
$y_g(s),\,y_{g^{-1}}(s)$
  involved in these representations.  These fixed elements
   are called {\it $x$-representatives of level $s$} and are
denoted by
 $ x_{g}(s)$
(under the Index Condition). The total number of
$x$-representatives of level  $s$ is at most $2(n-1)m$.
Note that
$x$-representatives of level $s$, elements
 $x_g(s)$, are also centralizers of level
 $s$.

It is clear from the construction that
\begin{equation}\label{f1-2} A_g(k+1)\leq A_g(k)\end{equation} for all $g\in
G\setminus\{e\}$ and any $k$.

By definition, if $ g_1\ldots g_{t-1}\,g\, g_{t} \ldots g_k=e$,
for some $t\leq k+1$ and   $k\leq W_s$, then a centralizer
$y_g(s)$ has the following property with respect to
representatives of lower levels:
\begin{equation} \label{f1-3} z_{g_1}(\varepsilon_1)\, \ldots\,
z_{g_{t-1}}(\varepsilon_{t-1})\,y_g(s)\,z_{g_t}(\varepsilon_t)\,\ldots\,
 z_{g_k}(\varepsilon_k)\in I_e,
 \end{equation}
where the elements $z_{g_j}(\varepsilon_j)$ are representatives
(that is, either $b_{g_j}(\varepsilon_j)$
 or $x_{g_j}(\varepsilon_j)$, in any combination) of any (possible different)
levels  $\varepsilon_l<s$. \vskip2ex

The following lemmas are direct consequences of the inclusions
(\ref{f1-2}), (\ref{f1-3}) and the definitions of representatives.
\vskip2ex

\begin{lemma}\label{l2} Let $g\in G\setminus \{e\}$. Then

1) every homogeneous element $a_e\in A_e$ can be represented
modulo $I_e$ as a linear combination of  representatives $x_e(0)$
of level~$0$.

 2) every product $a_g\, b_{g^{-1}}$ in homogeneous elements
can be represented modulo $I_e$ as a linear combination of
products of the same pattern in representatives of level~$0$.

3) every product  $y_{g}(k_1)y_{g^{-1}}(k_2)$  in centralizers of
levels  $k_1,k_2$ can be represented modulo $I_e$ as a linear
combination of products
 $x_{g}(s)x_{g^{-1}}(s)$ of the same pattern in
$x$-representatives of any level $s$ satisfying  $ 0\leq s\leq
\min \{ k_1,k_2\}$.
\end{lemma}

 \begin{lemma}\label{l4}  Let $y_{g}(l+1)$ be a centralizer of
  level $l+1$, $b_{h}(l)$ be $b$-representative of level $l$ with $g,h, gh\in G\setminus e$. Then elements of the form $u_{gh}=y_{g}(l+1)b_{h}(l)$ or  $v_{hg}=b_{h}(l)y_{g}(l+1)$
  are
 centralizers of level
$l$. \end{lemma}

\begin{proof} The proof follows directly from (\ref{f1-3}) and the definitions of $W_i$.
\end{proof}

\begin{lemma}\label{l5} Any product of the form $a_{g^{-1}}\,y_g(k+1)$ or
$y_g(k+1)\,a_{g^{-1}}$, where $y_g(k+1)$ is a centralizer of level
 $k>0$,  is equal modulo $I_e$ to a  product of the form $y_{g^{-1}}(k)\,y_g(k)$ or
 accordingly $y_g(k)\,y_{g^{-1}}(k)$,  where
$y_{g^{-1}}(k), \, \,y_g(k)$ are  centralizers of level
$k-1$.\end{lemma}

\begin{proof} We represent  $a_{g^{-1}}$ as a sum of a centralizer $y_{g^{-1}}(k)$ of level
$k$ and a linear combination  of $b$-representatives
$b_{g^{-1}}(k)$ of level $k$ and substitute this sum into the
product $a_{g^{-1}}\,y_g(k+1)$.  We obtain a sum of the element
$y_{g^{-1}}(k)\,y_g(k+1)$ and a linear combination of elements of
the form $b_{g^{-1}}(k)\,y_g(k+1)$. By~(\ref{f1-3}) the product
$b_{g^{-1}}(k)\,y_g(k+1)$ belongs to $I_e$.  Hence
$a_{g^{-1}}\,y_g(k+1)=y_{g^{-1}}(k)\,y_g(k+1)\,(\mathrm{mod\,}I_e).$
Similarly,
$y_{g^{-1}}(k+1)\,a_{g}=y_{g^{-1}}(k+1)\,y_g(k)\,(\mathrm{mod\,}I_e).$
Since $A_g(k)\geq A_g(k+1)$, both products  have the required
form.
\end{proof}

{\bf Construction of nilpotent ideal.}  Recall that $N=d^2+3$ is
the fixed notation for the highest level. We have constructed the
generalized centralizers $A_g(N)$ for
 $g\in G\setminus \{e\}$. Let $G\setminus \{e\}=\{g_1,\ldots, g_{n-1}\}$
 We set
$$
Z=_{\rm id}\left< A_{g_1}(N),\,A_{g_2}(N),\ldots ,A_{g_{n-1}}(N),
I_e\right>.$$
 This ideal  has $(n,d,m)$-bounded codimension in  $A$,
since each subspace  $A_h(N)$, $h\in G\setminus \{e\}$,  has
$(n,d,m)$-bounded codimension in $A_h$, while the dimension of
$A_e/I_e$ is at most $m$ by hypothesis. To prove the
Theorem~\ref{th2} we show  that the ideal $Z$ is nilpotent of
$(n,d)$-bounded class.
 \vskip2ex

{\bf Definition.} For every  $g\in G$ we set $Z_g=Z\cap A_g$.
\vskip1ex

\begin{lemma} \label{l6}
  The subspace $Z_e$ is contained modulo $I_e$ in the subspace
spanned by  products of the form $y_{h^{-1}}(N-2)\,y_h(N-2)$  and
by products of the form $a_{g^{-1}}\,i_e\, a_{g},$ where
$y_{h^{-1}}(N-2)$, $y_h(N-2)$ are centralizers of level $N-2$,
$a_{g^{-1}}\in A_{g^{-1}}$, $a_{g}\in A_{g}$, $i_e\in I_e$, $h,g
\in G\setminus \{e\}$.
\end{lemma}

\begin{proof} An element of $Z_e$ is modulo $I_e$ a linear combination of products of the  forms:
\begin{equation} \label{f2-1} a_{g^{-1}}\,i_e\, a_{g}, \,\,\,\mathrm{
where}\,\,\, a_{g^{-1}}\in A_{g}, \,\, i_e\in I_e,\,\, a_{g}\in
A_{g},\,\,g\neq e \{e\}\end{equation}
\begin{equation} \label{f2-2}a_{g^{-1}}\,y_g(N),
\,\,\,\mathrm{ where}\,\,\, a_{g^{-1}}\in A_{g^{-1}}, \,\, g\neq
e,\,\,\,y_{g}(N)\in A_{g}(N),\,\, \end{equation}
\begin{equation} \label{f2-3} y_g(N)\,a_{g^{-1}},\,\,\,\mathrm{
where}\,\,\,\,\, g\neq e,\,\,\, y_{g}(N)\in A_{g}(N), \,\,
a_{g^{-1}}\in A_{g^{-1}},\,\, \end{equation}
\begin{equation} \label{f2-4} a_{g_1}\,y_g(N)\,a_{g_{2}},\,\,\,\mathrm{
where}\,\,\,a_{g_1}\in A_{g_1}, \,\, a_{g_2}\in A_{g_2},\,\,
y_g(N)\in A_g(N),\,\,\, g_1gg_2=e,\end{equation} The product
(\ref{f2-1}) is already of the required form. By Lemma~\ref{l5}
the products $y_g(N)\,a_{g^{-1}}$ and $a_{g^{-1}}\,y_g(N)$ can be
represented modulo $I_e$ as linear combinations of products of the
form $y_g(N-1)\,y_{g^{-1}}(N-1)$ and therefore have also the
required representation since $A_g(N-1)\leq A_g(N-2)$.

Consider  the product (\ref{f2-4}).   Since $g_1gg_2=e$ and $g\neq
e$, at least one $g_i$, $i=1, 2$ is not equal to  $e$. Let, for
example, $g_1\neq e$. We represent $a_{g_1}$  as a sum of  a
centralizer $y_{g^{-1}}(N-1)$ of level $N-1$ and a linear
  combination of
$b$-representatives $b_{g^{-1}}(N-1)$ of level $N-1$ and insert
this expression into (\ref{f2-4}). We obtain a linear combination
of products of the following two forms
\begin{equation} \label{f2-5} y_{g_1}(N-1)\,y_g(N)\,a_{g_2} \end{equation} and
\begin{equation} \label{f2-6} b_{g_1}(N-1)\,y_g(N)\,a_{g_2}. \end{equation}
In (\ref{f2-5}) we set $a_{g_1^{-1}}:=y_g(N)\,a_{g_2}$.   Applying
Lemma \ref{l5} and the inclusions~(\ref{f1-2}) to
$y_{g_1}(N-1)a_{g_1^{-1}}$ we obtain that (\ref{f2-5}) is equal
modulo $I_e$ to a product of the required form
$y_{g_1}(N-2)\,y_{g_1^{-1}}(N-2)$.

Let us now consider  the product (\ref{f2-6}). If $g_2=e$, then
$g_1g=e$ and $b_{g_1}(N-1)\,y_g(N)\in I_e$ by~(\ref{f1-3}). Since
$I_e$ is an ideal of $A_e$ and $g_2=e$,
$$b_{g_1}(N-1)\,y_g(N)\,a_{e}\in I_e.$$
If $g_2\neq e$, then $g_1g\neq e$ and $b_{g_1}(N-1)\,y_g(N)$ is a
 a centralizer of level $N-1$ by Lemma~\ref{l4}:
$$b_{g_1}(N-1)\,y_g(N)\,a_{g_2} =
y_{g_1g}(N-1)\,a_{g_2}.$$
Again by Lemma \ref{l5} the product
$y_{g_1g}(N-1)\,a_{g_2}$ is equal modulo $I_e$ to the product of
the require form $y_{g_1g}(N-2)\,y_{g_2}(N-2)$.  The case where
$g_1=e, g_2\neq e$ in (\ref{f2-4}) can be treated in the same
manner.
\end{proof}

\vskip2ex

{\it Proof of Theorem~{\rm \ref{th2}}.\/} We set $H=d^2+1$,
$T=d(H-1)+1=d^3+1$, \, $S=(T-1)(n-1)+1=d^3(n-1)+1$, \,
$U=d(n-1)$,\, and
$Q=(U+1)(S-1)+1=(d(n-1)+1)d^3(n-1)=d^4(n-1)^2+d^3(n-1)$.
 By Lemma~\ref{l-b-i-2} it suffices to show that $ (Z_{e})^{Q}=0$.

We consider an arbitrary  product of length $Q$ in elements $c_i$
from $Z_e$:
\begin{equation}\label{f3-1} c_1\,c_2 \ldots
c_Q,\end{equation} (here the indices are numbering). By
Lemma~\ref{l6} we can represent modulo $I_e$ every $c_k$ as a
linear combination of products of some special form. Substituting
these expressions into (\ref{f3-1}) we obtain a linear combination
of elements
\begin{equation}\label{f3-2} z_1\,z_2\,\ldots z_Q,\end{equation} where the $z_k$ (here the indices are also numbering) are
either elements $i_e\in I_e$ or products  $c_e=a_{g^{-1}}\,w_e\,
a_{g}\in A_e$, $w_e\in I_e$ or products  $v_e=y_{g_k^{-1}}(N-2)\,
y_{g_k}(N-2)\in A_e,$ in centralizers $y_{g_k^{-1}}(N-2),\,
y_{g_k}(N-2)$  of level $N-2$.

If in (\ref{f3-2}) among $z_k$ there are at least $d$ occurrences
of elements $i_e\in I_e$ the summand is trivial, since $I_e$ is an
ideal of $A_e$ and $(I_e)^d=0$.

Suppose now that  in (\ref{f3-2}) there are at least $(d-1)n+1$
entries of products  $c_e=a_{g^{-1}}\,i_e\, a_{g}$. Among them we
can choose
 $d$ products $c_e=a_{g_k^{-1}}\,i_e\, a_{g_k}\in A_e$ with the same pair of
indices $g_k^{-1}, \,g_k$:
$$z_1\,\ldots z_{l_1}\,\underbrace{a_{g_k^{-1}}\,i_e\, a_{g_k}}\,z_{l_1+1}\ldots z_{l_2}\,
\underbrace{a_{g_k^{-1}}\,i_e\, a_{g_k}}\,z_{l_2} \ldots
z_{l_k}\,\underbrace{a_{g_k^{-1}}\,i_e\, a_{g_k}}\,
z_{l_k+1}\ldots z_Q .$$ Since the products
$a_{g_k}\,z_{l_s+1}\ldots z_{l_{s+1}}\, a_{g_k^{-1}}$ between the
elements $i_e$ belong to $A_e$,  $I_e$ is an ideal in $A_e$ and
$(I_e)^d=0$, then the  product (\ref{f3-2}) is equal to $0$.

Consider the case where      the number of $i_e$-occurrences in
(\ref{f3-2}) is at most $d$ and the number of $c_e$-occurrences is
at most $U=d(n-1)$. Since $Q= (U+d+1)(S-1)+1$, the product
(\ref{f3-2}) has at least one subproduct  consisting of $S$
elements $v_e$ (going one after another):
\begin{equation}\label{f3-3}\big(y_{g_1^{-1}}(N-2)\,
y_{g_1}(N-2)\big)\,\,\big(y_{g_2^{-1}}(N-2)\, y_{g_2}(N-2)\big)
\ldots \big(y_{g_{S}^{-1}}(N-2)\, y_{g_{S}}(N-2)\big),
\end{equation} where\, $y_{g_i}(N-2)\in A_{g_i}(N-2)$,  $y_{g_i^{-1}}(N-2)\in A_{g_i^{-1}}(N-2)$ are (possibly different)
centralizers of level $N-2$.  Since $S=(T-1)(n-1)+1$ in
(\ref{f3-3}) there are at least $T$ entries of products
$y_{g_i^{-1}}(N-2)y_{g_i}(N-2)$ with the same pair of indices,
say, $g_k^{-1}, \,g_k$. We choose any $T$ such products and
represent  modulo $I_e$  all the other pairs as linear
combinations of  products of representatives of level $0$     by
Lemma \ref{l2}:
\begin{equation}\label{f3-4}w_e\ldots\, w_e \big(y_{g_k^{-1}}(N-2)\,
y_{g_k}(N-2)\big)\,w_e\ldots w_e\big(y_{g_k^{-1}}(N-2)\,
y_{g_k}(N-2)  \big) \ldots ,  \end{equation} where there are $T$
occurrences of (possibly different)  products
$y_{g_k^{-1}}(N-2),\, y_{g_k}(N-2)$ with the  same pair of indices
$g_k^{-1}, \,g_k$, the $w_e$ are possibly different elements of
$A_e$: either $i_e\in I_e$ or representatives $x_e(0)$. If in
(\ref{f3-4}) among $w_e$ there are at least $d$ occurrences of
elements of $I_e$ the summand is trivial, since $I_e$ is an ideal
of $A_e$ and $(I_e)^d=0$. In the opposite case, as $T=d(H-1)+1$,
there is a subproduct  of the form
$$\big(y_{g_k^{-1}}(N-2)\, y_{g_k}(N-2)\big)\,x_e(0)\ldots
x_e(0)\big(y_{g_k^{-1}}(N-2)\, y_{g_k}(N-2)  \big) \ldots,
$$
where there are $H=d^2+1$ occurrences of products
$y_{g_k^{-1}}(N-2)\, y_{g_k}(N-2)$ and between them there are only
$x$-representative of level $0$ and no elements from $I_e$. By
lemma \ref{l2} we represent modulo $I_e$ the first entry
$y_{g_k^{-1}}(N-2)\, y_{g_k}(N-2)$ as a linear combination of the
products of the same pattern in representatives  in level $1$, the
second
--- in level $2$, and so on, the last one --- in level $H$. We
obtain  a linear combination
$$\big(x_{g_k^{-1}}(1)\, x_{g_k}(1)+i_e\big)\,x_e(0)\ldots
x_e(0)\big(x_{g_k^{-1}}(2)\, x_{g_k}(2)+i_e  \big) \ldots
\big(x_{g_k^{-1}}(H)\, x_{g_k}(H)+i_e\big).$$
Expanding this expression we get  a linear combination of products
of the form
$$c_1\,x_e(0)\ldots
x_e(0)\,c_2\, x_e(0)\ldots x_e(0)\,\ldots c_H,$$ (here the indices
are numbering) where the $c_k$ are either elements $i_e\in I_e$ or
products $x_{g_k^{-1}}(n_k)\, x_{g_k}(n_k)$ of different levels
with one and the same pair of indices $g_k^{-1},\,g_k\in G$. If in
a summand there are at least $d$ entries of $i_e\in I_e$ it is
trivial by assumptions. In the summands with  less than $d$
entries of the $i_e$, we can find
 an interval long enough without $i_e$-entries.  More precisely, since $H=d^2+1$
 there is a subproduct of the form
$$\big(x_{g_k^{-1}}(s)\, x_{g_k}(s)\big)\,x_e(0)\ldots
x_e(0)\big(x_{g_k^{-1}}(s+1)\, x_{g_k}(s+1)\big) \ldots
\big(x_{g_k^{-1}}(s+d)\, x_{g_k}(s+d)\Big),$$
where there are $d+1$  products $y_{g_k^{-1}}(l)\, y_{g_k}(l)$ of
different levels $l=s,\ldots, s+d$. For each $t=0,\ldots,d-1$, the
product $$ x_{g_k}(s+t)\,x_e(0)\ldots
x_e(0)\,x_{g_k^{-1}}(s+t+1)$$ includes exactly one centralizer of
level $s+t+1$, all the other elements are representatives of lower
levels, and the weight of the product is at most
$2S=2d^3(n-1)+2=W_1\leq W_{s+t+1}$. By  (\ref{f1-3})
$$ x_{g_k}(s+t)\,x_e(0)\ldots
x_e(0)\,x_{g_k^{-1}}(s+t+1)\in I_e$$ for all $t=0,\ldots,d-1$. It
follows that (\ref{f3-7}) is equal to product
 $$ x_{g_k^{-1}}(s)
\underbrace{(i_e\, i_e\, \ldots\, i_e)}_{d}\,x_{g_k}(s+d)=0,$$
which is trivial, since $(I_e)^d=0$. \ep

\section{Proof of the main result}\label{proof-theorems}

In this section we prove Theorem~\ref{th1}. Recall that we are
given an associative algebra $A$ over a field $F$ that admits a
finite soluble automorphism group $G$ of order $n$ prime to the
characteristic of $F$ such that the fixed-point  subalgebra  $A^G$
has a two-sided nilpotent ideal $I\lhd A^G$ of nilpotency index
$d$ and of finite codimension  $m$ in $A^G$. The aim is to  find a
nilpotent ideal in $A$ of ($n,d)$-bounded nilpotency index  and of
finite $(n,d,m)$-bounded  codimension.

{\it Proof of Theorem~{\rm \ref{th1}}.\/}  First, we consider  the
case where  $G$ is a cyclic group of prime order $p$.  Let  $g$ be
a generator of $G$. Then $g$ induces an automorphism of the
algebra $A\otimes _{\Z }\, \Z [\w ]$, where $\w$ is a primitive
$p$-th root of unity. The fixed-point subalgebra  of this
automorphism denoted by the same letter has the same dimension $m$
over the field extended by~$\w$.    It suffices to prove
Theorem~\ref{th1} for the algebra $A\otimes _{\Z }\, \Z[\w ]$.
Hence in what follows we can assume  that the ground field $F$
contains~$\w$.  We define the homogeneous {\it compo\-nents\/}
$A_k$ for $k=0,\ldots, p-1$ as the subspaces
$$A_k=\left\{ a\in A\mid a^g=\w ^{k}a\right\} .$$
Since the characteristic of $F$ does not divide $p$, we have
$$A=A_0\oplus A_1\oplus \cdots \oplus A_{p-1}.$$ This
decomposition determines a grading on $A$ by a cyclic group of
prime order $p$, with $A_0=A^G$ in view of the obvious inclusions
$$A_sA_t\subseteq A_{s+t},$$ where $s+t$ is
computed modulo~$p$. Hence  the case $|G|=p$ in Theorem~\ref{th1}
follows from Theorem~\ref{th2}.

Let now $G$ be any finite soluble group of automorphisms of $A$,
and suppose that its order $n$ is not divisible by the
characteristic of $F$ . We use induction on $|G|$. We may assume
that $n$ is not a prime number. This means that in $G$ there is a
non-trivial normal subgroup $H$. We consider the subalgebra
$C=A^H$ of its fixed points. Since $H\vartriangleleft G$, we have
$C^g\leqslant C$  for any $g\in G$.  The subalgebra $C $ admits a
finite solvable group of automorphisms of order $\leq |G/H|$ which
is strictly less than $|G|$ and not divisible by the
characteristic of $F$. By induction $C$ has a nilpotent ideal
$J\lhd C$ of $(|G/H|,d, m)$-bounded codimension~$t=t(|G/H|,d, m)$
and of $(|G/H|,d)$-bounded nilpotency index $h=h(|G/H|,d)$. By
Lemma~\ref{l4-2} there exists a nilpotent $G$-invariant ideal
$K\geq J$ in $A$ of nilpotency index $h_1=h_1(|H|,h)$ bounded by
$|H|$ and by the nilpotency index of $J$. The subgroup $H$ acts on
the factor-algebra $\bar A=A/K$ and subalgebra of fixed points
$\bar A^{H}$ has dimension at most $t$. We  apply induction
hypothesis to the algebra $\bar A$ and the automorphism group $H$
of $\rm{Aut}\,\bar A$ whose order is strictly less than $|G|$. The
algebra $\bar A$ has a nilpotent ideal $Z$ of $(|H|,t)$-bounded
codimension and of $|H|$-bounded nilpotent index $h_2=h_2(|H|)$.
The image of $Z$ in $A$ is a required ideal since its nilpotency
index is at most $h_1h_2$, which is a $(n,d)$-bounded number, and
the codimension is $(n,d,m)$-bounded. \ep

\end{document}